\documentclass[a4paper,12pt]{article}

\usepackage[T1]{fontenc}

\usepackage{amsmath, graphicx, color, amsthm, centernot, amssymb}
\usepackage[font=small,labelfont=bf,width=12.5cm]{caption}
\usepackage{algorithmic, algorithm}

\usepackage{fullpage}

\newtheorem{claim}{Claim}
\newtheorem{theorem}{Theorem}
\newtheorem{lemma}{Lemma}

\newtheorem{proposition}{Proposition}
\newtheorem{conjecture}{Conjecture}

\newtheorem{observation}{Observation}

\newcommand{\chiC}[1]{\chi_{\mathrm{fum}}(#1)}
\newcommand{\chiCe}[1]{\chi_{\mathrm{fum}}'(#1)}

\newenvironment{proofclaim}[1][]%
    {\noindent \emph{Proof.} {}{#1}{}}{$~$\hfill $~\blacklozenge$ \vspace{0.2cm}}

\title{On facial unique-maximum (edge-)coloring}
\author
{
	Vesna Andova\thanks{Faculty of Electrical Engineering and Information Technologies,
Ss Cyril and Methodius Univ., Skopje, Macedonia \& Faculty of Information Studies, Novo mesto, Slovenia.
		E-Mail: \texttt{vesna.andova@gmail.com}}, \
	Bernard Lidick\'y\thanks{Department of Mathematics, Iowa State University, USA.
		E-Mail: \texttt{lidicky@iastate.edu}}, \		
	Borut Lu\v{z}ar\thanks{Faculty of Information Studies, Novo mesto, Slovenia.
		E-Mail: \texttt{borut.luzar@gmail.com}}, \
  	Riste \v{S}krekovski\thanks{Faculty of Information Studies, Novo mesto
  		\& University of Ljubljana, Faculty of Mathematics and Physics
  		\& University of Primorska, FAMNIT, Koper, Slovenia.
   		E-Mail: \texttt{skrekovski@gmail.com}}		
}

\begin{document}
\maketitle

{
	\abstract
	{
		A \textit{facial unique-maximum coloring} of a plane graph is a vertex coloring where
		on each face $\alpha$ the maximal color appears exactly once on the vertices of $\alpha$.
		If the coloring is required to be proper, then the upper bound for the minimal number of colors required for such a coloring is set to $5$.
		Fabrici and G\"{o}ring~\cite{FabGor16} even conjectured that $4$ colors always suffice.
		Confirming the conjecture would hence give a considerable strengthening of the Four Color Theorem.
		In this paper, we prove that the conjecture holds for subcubic plane graphs, outerplane graphs and
		plane quadrangulations.
		Additionally, we consider the facial edge-coloring analogue of the aforementioned coloring and
		prove that every $2$-connected plane graph admits such a coloring with at most $4$ colors.
	}

	\bigskip
	{\noindent\small \textbf{Keywords:} facial unique-maximum coloring, facial unique-maximum edge-coloring, plane graph.}
}

\section{Introduction}

In this paper, we consider simple graphs only.
We call a graph \emph{planar} if it can be embedded in the plane without crossing edges
and we call it \emph{plane} if it is already embedded in this way.
A \emph{coloring} of a graph is an assignment of colors to vertices.
If in a coloring adjacent vertices receive distinct colors, it is \emph{proper}.
The cornerstone of graph colorings is the Four Color Theorem stating that every planar graph
can be properly colored using at most $4$ colors~\cite{AppHak76}.
Fabrici and G\"{o}ring~\cite{FabGor16} proposed the following strengthening of the Four Color Theorem.

\begin{conjecture}[Fabrici and G\"{o}ring \cite{FabGor16}]
	\label{conj:plane4}
	If $G$ is a plane graph, then there is a proper coloring of the vertices of $G$
	by colors in $\{1,2,3,4\}$ such that every face contains a unique vertex colored with
	the maximal color appearing on that face.
\end{conjecture}

A proper coloring of a graph embedded on some surface, where colors are integers and every face has a unique vertex colored with a maximal color,
is called a \emph{facial unique-maximum coloring} or \emph{FUM-coloring} for short (Wendland uses the notion \emph{capital coloring} instead).
This type of coloring was first studied by Fabrici and G\"{o}ring~\cite{FabGor16}.
The main motivation for their research comes from the \emph{unique-maximum coloring} (also known as \emph{ordered coloring}), defined as a coloring
where there is only one vertex colored with the maximal color on every path in a graph.
Studying unique-maximum coloring was motivated due to a number of applications it finds in various branches of
mathematics and computer science; see, e.g.,~\cite{CheKesPal13,CheTot11,KatMccSea95} for more details.
Fabrici and G\"{o}ring used this concept in a facial version, which is of great interest, among others,
also due to Conjecture~\ref{conj:plane4} and its direct connection to the Four Color Theorem.
Coloring embedded graphs with respect to faces is a bursting field itself; the main directions are presented in a recent survey by Czap and Jendrol'~\cite{CzaJen16}.

For a graph $G$, the minimum number $k$ such that $G$ admits a FUM-coloring
with colors $\{1,2,\ldots,k\}$ is called the \emph{FUM chromatic number of $G$}, denoted by $\chiC{G}$.
Fabrici and G\"{o}ring~\cite{FabGor16} proved that if $G$ is a plane graph, then  $\chiC{G} \leq 6$.
Their result was further improved as follows.
\begin{theorem}[Wendland~\cite{Wen16}]
	\label{thm:plane5}
	If $G$ is a plane graph, then $\chiC{G} \leq 5$.
\end{theorem}

We show that the upper bound $4$ from Conjecture~\ref{conj:plane4} holds for several subclasses of plane graphs, and that,
surprisingly, the bound is tight in most of the cases. The main result of the paper regarding the FUM-coloring of vertices is the following.
\begin{theorem}
	\label{thm:subcubic}
	If $G$ is a plane subcubic graph or an outerplane graph, then $\chiC{G} \leq 4$.
\end{theorem}

In the second part of the paper, we consider the edge-coloring version of the problem,
which has been introduced by Fabrici, Jendrol', and Vrbjarov\'{a}~\cite{FabJenVrb15}.
For a graph $G$ embedded on some surface, two distinct edges are said to be \emph{facially adjacent}
if they are consecutive in some facial path, i.e., they have a common vertex and they are incident with a same face.
A \emph{facial edge-coloring} is a coloring of edges such that facially adjacent edges receive distinct colors.
It is rather straightforward to prove that every plane graph admits a facial edge-coloring with at most $4$ colors.

For a graph $G$, we denote by $\chiCe{G}$ the minimum number $k$ such that
there exists a facial edge-coloring using colors $1,\ldots,k$
such that each face is incident with a unique edge colored with the maximal color.
Such a coloring is called a \emph{FUM-edge-coloring}.
In~\cite{FabJenVrb15}, Fabrici et al. proposed the following conjecture.

\begin{conjecture}[Fabrici et al.~\cite{FabJenVrb15}]
	\label{conj:main}
	If $G$ is a $2$-edge-connected plane graph, then $\chiCe{G} \leq 4$.
\end{conjecture}

In~\cite{FabJenVrb15}, the authors proved that $\chiCe{G} \leq 6$ for every $2$-edge-connected plane graph $G$.
Our main result is that we prove $\chiCe{G} \leq 4$ if the assumption that the graph is $2$-edge-connected
is replaced by $2$-vertex-connectivity, supporting Conjecture~\ref{conj:main}.

\begin{theorem}
	\label{thm:maine}
	If $G$ is a $2$-vertex-connected plane graph, then $\chiCe{G} \leq 4$.
\end{theorem}

Observe that every edge in an embedded graph is facially adjacent to at most four other edges,
therefore one can translate the problem of facial edge-coloring of a plane graph
to a vertex coloring of a plane graph with maximum degree $4$.
Hence, Theorem~\ref{thm:plane5} directly implies $\chiCe{G} \leq 5$ for every plane graph $G$.
Similarly, Theorem~\ref{thm:subcubic} implies that if $G$ is obtained from a plane graph by subdividing every edge, then $\chiCe{G} \leq 4$.

\smallskip
The paper is organized as follows.
In Section~\ref{sec:v}, we prove Theorem~\ref{thm:subcubic} and discuss the FUM-coloring of vertices.
In Section~\ref{sec:e}, we consider the FUM-edge-coloring and prove Theorem~\ref{thm:maine}.
Both proofs, of Theorem~\ref{thm:subcubic} and Theorem~\ref{thm:maine}, use precoloring extension technique
successfully applied by Thomassen~\cite{Tho94} when proving that every planar graph is $5$-choosable.
In Concluding remarks, we present some related results and discuss possible future directions on this topic.

\section{FUM-(vertex-)coloring}
\label{sec:v}

In this section we consider the FUM-coloring of vertices and confirm that Conjecture~\ref{conj:plane4} holds for
several subclasses of plane graphs.

First, we recall a theorem, which is the main tool used in~\cite{FabGor16}, and will prove helpful also in proving our results.
\begin{theorem}[Fabrici and G\"{o}ring \cite{FabGor16}]
	\label{thm:aux}
	Every plane graph has a (not necessarily proper) $3$-coloring with colors black, blue, and red such that
	\begin{itemize}
		\item[(1)] each face is incident with at most one red vertex,
		\item[(2)] each face that is not incident with a red vertex is incident with exactly one blue vertex.
	\end{itemize}
\end{theorem}
A slightly stronger version of Theorem~\ref{thm:aux} was proved by Wendland~\cite{Wen16} who
also added the conclusion that each triangle, facial or separating, contains at least one vertex that is not black.
This enabled him to improve the upper bound to $5$ colors.

Recall that Conjecture~\ref{conj:plane4} states that if $G$ is a plane graph, then its FUM chromatic number is 4,
which is the same upper bound as for the chromatic number.
One can therefore ask, which are the plane graphs admitting a FUM-coloring with at most $3$ colors.
However, natural candidates such as graphs of large girth, quadrangulations, and outerplane graphs
have infinitely many examples with FUM chromatic number $4$.

The example in Figure~\ref{fig:girth-v} shows that there is no analogue of Gr\"{o}tzsch's result for the FUM-coloring.
Indeed, every vertex lies on the outer face, and hence only one can be colored with $3$ (assuming $3$ colors suffice). As every vertex is incident to at
most three faces, the maximal color of the fourth face is $2$, and hence all the other vertices should receive $1$,
which is not possible, since the coloring must be proper.
\begin{figure}[htp!]
	\begin{center}
		\includegraphics{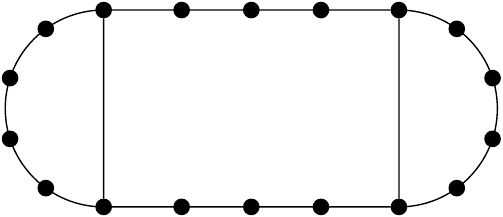}
	\end{center}
	\caption{Plane graphs with arbitrarily large girth (in fact also outerplane graphs) need at least $4$ colors for a FUM-coloring.}
	\label{fig:girth-v}
\end{figure}

We continue by considering plane quadrangulations.
\begin{proposition}
	If $G$ is a plane quadrangulation, then $\chiC{G} \leq 4$.
	Moreover, there exists an infinite family of plane quadrangulations with FUM chromatic number $4$.
\end{proposition}

\begin{proof}
	Let $G$ be a plane quadragulation.
	A FUM-coloring of $G$ with at most $4$ colors can be obtained by using Theorem~\ref{thm:aux} to assign the colors $3$ and $4$
	such that every face is incident with at most one $4$, and at most one $3$ if it is not incident with $4$;
	the remaining vertices may be colored by $1$ and $2$, since $G$ is bipartite.
	
	To prove the second part of the proposition, consider the graph $H$ depicted in Figure~\ref{fig:quadrangulation}.	
	\begin{figure}[ht]
		$$
			\includegraphics{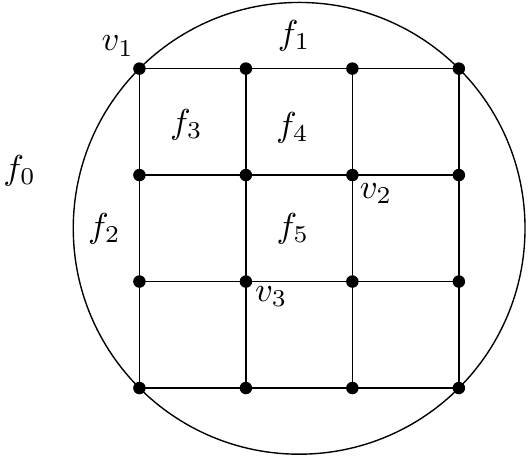}
		$$
		\caption{A plane quadrangulation with FUM chromatic number $4$.}
		\label{fig:quadrangulation}
	\end{figure}
	Suppose $\chiC{H} = 3$. Then, one of the vertices incident with the outer face $f_0$, say $v_1$, must be colored with $3$.
	This sets the maximal color also for the faces $f_1$, $f_2$, and $f_3$. Thus, to provide a unique maximal color for $f_4$,
	we must color the vertex $v_2$ with $3$. Analogously, we must color with $3$ also the vertex $v_3$. 	
	But now, there are two vertices colored with $3$ incident with $f_5$, a contradiction.
	
	One obtains an infinite family of graphs that require $4$ colors, e.g., by inserting a copy of $H$ 
	to the face $f_5$ by gluing the edges of the outer face of $H$ and the edges of $f_5$.
\end{proof}

We establish Conjecture~\ref{conj:plane4} also for the classes of subcubic plane graphs and outerplane graphs.
The following lemma is motivated by Theorem~\ref{thm:aux}, and we use it to prove Theorem~\ref{thm:subcubic}.
The upper bound of $4$ is tight for both classes by, e.g., the graph in Figure~\ref{fig:girth-v}.

\begin{lemma}
	\label{lem:sub}
	Suppose $G$ is a plane graph that is either subcubic or outerplane,
	$P$ is a path in the outer face of $G$ on at most two vertices,
	and the vertices of $P$ are properly colored by a coloring $c'$ with colors $\{1,2,3\}$.
	Then there is a vertex coloring $c$ of $G$ with at most $4$ colors such that
	\begin{itemize}
		\item $c$ matches $c'$ on $P$,
		\item $c(v) \in \{1,2,3\}$ if $v$ is incident with the outer face, and
		\item each inner face has a vertex with unique maximal color.
	\end{itemize}
\end{lemma}

\begin{proof}
	Let $G$ be a smallest counterexample in terms of the number of vertices and with largest path $P$.
	Clearly, we may assume $G$ has at least $2$ vertices.
	If $G$ has more than one component incident with the outer face, 
	then, by the minimality of $G$, for each of these components, we can color the subgraph induced by the component 
	and the vertices in its interior.	
	The colorings of all such subgraphs together give us a required coloring of $G$, a contradiction.
	Hence, we may assume that precisely one component of $G$ is incident with the outer face.
	If $P$ has less than two vertices, we extend $P$ arbitrarily by coloring one of its neighbors on the outer face.
	Hence $P$ has two vertices.

	We split the rest of the proof into four claims.
\begin{claim}
	The outer face of $G$ is bounded by a cycle.
\end{claim}

\begin{proofclaim}
	Suppose for a contradiction that $v$ is a cut-vertex in $G$ incident with the outer face.
	Let $W$ be the set of vertices consisting of $v$ and the vertices of the connected component of $G-v$ that intersects $P$.
	Let $X = (V(G) \setminus W) \cup \{v\}$.
	By the minimality of $G$, there exists a coloring $c_W$ of $G[W]$ with a path $P_W = P$ and a coloring $c_W'=c'$,
	and there exists a coloring $c_X$ of $G[X]$ with $P_X=\{v\}$ and $c_X'$ being $c_W$ restricted to $v$.
	Since the colorings $c_W$ and $c_X$ assign the same color to $v$, they can be combined into a coloring $c$ of $G$, a contradiction.
\end{proofclaim}

Denote by $C$ the cycle by which the outer face of $G$ is bounded.

\begin{claim}
	\label{cl:chord}
	$C$ has no chords.
\end{claim}

\begin{proofclaim}
	Suppose for a contradiction that $uv$ is a chord in $C$.
	Let $W$ be the set of vertices containing $u$, $v$, and the vertices of the connected component of $G - \{u,v\}$ that intersects $P$.
	Let $X = (V(G) \setminus W) \cup \{u,v\}$.
	By the minimality of $G$, there exists a coloring $c_W$ of $G[W]$ with $P_W = P$ and $c_W'=c'$, and
	there exists a coloring $c_X$ of $G[X]$ with $P_X=\{u,v\}$ and $c_X'$ being $c_W$ restricted to $u$ and $v$.
	Since the colorings $c_W$ and $c_X$ assign the same colors to $u$ and $v$,
	they can be combined into a coloring $c$ of $G$, a contradiction.
	Hence $C$ is a chordless cycle.
\end{proofclaim}

If $G$ is outerplane, it follows from Claim~\ref{cl:chord} that it must be a cycle.

\begin{claim}
	\label{cl:cycle}
	$G$ is not a cycle.
\end{claim}

\begin{proofclaim}
	Suppose for a contradiction that $G$ is a cycle.
	The coloring $c'$ assigns the color $3$ to at most one vertex of $P$.
	Hence it is possible to color the vertices of $G$ such that exactly one vertex $x$
	is colored with $3$ and all the others are colored with $1$ and $2$.
	The interior face of $G$ then has $x$ as the unique vertex colored by the maximal color.
\end{proofclaim}

Hence, $G$ is not outerplane, so it is subcubic.
Moreover, it contains at least one vertex, which is not in $C$; we call such vertices \emph{interior}.

\begin{claim}
	\label{mainclaim}
	In $V(C) \setminus V(P)$, there is no vertex of degree $3$  with an interior neighbor,
	nor a vertex of degree $2$ that is incident with a same face as any interior vertex.
\end{claim}

\begin{proofclaim}
	Suppose for a contradiction that $v \in V(C) \setminus V(P)$ is a vertex of degree $3$ with an interior neighbor $u$,
	or a vertex of degree $2$ and $u$ is an interior vertex incident with a same face as $v$.
	Let $G'$ be the graph obtained from $G$ by deleting $u$ and $v$.
	By the minimality of $G$, there is a coloring $c$ of $G'$ satisfying the assumptions of Lemma~\ref{lem:sub}.
	Notice that all the vertices incident with the same faces as $u$ in $G$ are incident with the outer face in $G'$ (except for $v$).
	Hence the neighbors of $u$ are colored by $c$ with the colors in $\{1,2,3\}$.
	We extend $c$ to $G$ by setting $c(u)=4$ and assigning to $v$ a color from $\{1,2,3\}$, which does not appear on its two neighbors on the outer face,
	a contradiction.
\end{proofclaim}

From Claim~\ref{mainclaim}, it follows that if $G$ is a subcubic plane graph,
there are only vertices of degree $2$ in $V(C) \setminus V(P)$. Moreover,
if there is an interior vertex in $G$, then it is incident with the same face as one
of the vertices in $V(C) \setminus V(P)$.
Hence, Claims~\ref{cl:cycle} and~\ref{mainclaim} give us a contradiction on existence of $G$.
This finishes the proof of Lemma~\ref{lem:sub}.
\end{proof}

Now, we are ready to prove the main theorem of this section.

\begin{proof}[Proof of Theorem~\ref{thm:subcubic}]
	Let $G$ be a plane subcubic graph or an outerplane graph and $v$ any vertex in the outer face of $G$.
	Apply Lemma~\ref{lem:sub} on the graph $G-v$ and color $v$ by $4$ to complete the coloring of $G$.
\end{proof}

\section{FUM-edge-coloring}
\label{sec:e}

In this section we turn our attention to the FUM-edge-coloring.
Notice that the upper bound of $4$ is the same as in the vertex version, and as already remarked,
the edge version is only a special case of the former. However, also here, the upper bound is achieved
within very particular classes of plane graphs, e.g., subcubic outerplane bipartite graphs of arbitrarily large girth
(see Figure~\ref{fig:girth-e} for an example).
\begin{figure}[htp!]
	\begin{center}
	\includegraphics{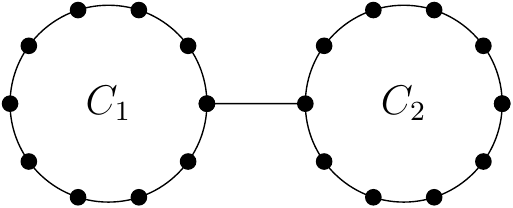}
	\end{center}
	\caption{Subcubic outerplane bipartite graphs of arbitrarily large girth need $4$ colors for FUM-edge-coloring.}
	\label{fig:girth-e}
\end{figure}
However, regarding Conjecture~\ref{conj:main}, Theorem~\ref{thm:maine} is the first result supporting it.

Let $G$ be a plane graph.
If an edge $e=uv$ is removed from $G$, new facial adjacencies of edges may be introduced around $u$ and $v$ in $G-e$.
However, if we are interested only in a facial edge-coloring of $G$,
these new adjacencies may be ignored when coloring $G-e$.
This motivates the following concept:
let $\mathcal{F}$ be a set of pairs of edges.
An \emph{$\mathcal{F}$-facial edge-coloring} is an edge-coloring, where every pair of facially adjacent edges
that are not in $\mathcal{F}$ receive distinct colors.
We call $\mathcal{F}$ the set of \emph{free pairs}.
Two edges are a \emph{good pair} if they are a free pair or if they have a vertex of degree $2$ in common.
If a vertex $v$ is a common vertex of the edges in a good pair, we call $v$ a \emph{good vertex}.

Recall that every graph $G$ can be decomposed into maximal 2-connected blocks.
The \emph{block graph} $B(G)$ is an intersection graph of blocks in $G$.
Notice that $B(G)$ is a tree and hence has at least two leaves (unless $G$ is $2$-connected).
We call a block corresponding a leaf a \emph{leaf-block}.

\begin{observation}
	\label{obs:leaves}
	Let $G$ be a $2$-connected graph.
	If $uv$ is an edge of $G$, then $\{u,v\}$ intersects the set of vertices of every leaf-block of $G-uv$.
\end{observation}

The following lemma is the core of the proof of Theorem~\ref{thm:maine}.

\begin{lemma}
	\label{lem:maine}
	Let $G$ be a plane graph and let $\mathcal{F}$ be a set of free pairs,
	where every leaf-block of $G$ has a good vertex in the outer face.
	Then there exists an $\mathcal{F}$-facial edge-coloring $c$ using colors in $\{1,2,3,4\}$ such that
	\begin{itemize}
		\item{} every edge in the outer face is colored with a color in $\{1,2,3\}$, and
		\item{} every face, except the outer face, has an edge of a unique maximal color.
	\end{itemize}
\end{lemma}

\begin{proof}
	Let $G$ be the smallest counterexample in terms of the sum of the number of vertices and edges.

	First we outline a process of removing an edge from $G$.
	Let $e=uv$ be an edge of $G$.
	Suppose $u$ is a vertex of degree at least $4$.
	Observe that in $G-e$, the edges $e_1$ and $e_2$ that were facially adjacent to $e$ at vertex $u$
	are not facially adjacent to each other in $G$, but they are facially adjacent in $G-e$.
	Hence, when considering $G-e$, we modify $\mathcal{F}$ by adding the pair $\{e_1,e_2\}$.
	This means $u$ is a good vertex in $G-e$. Similarly, $v$ is good, since it is either a common vertex of a free pair
	or it has degree at most $2$ in $G-e$.
	Hence, by Observation~\ref{obs:leaves}, every leaf-block in $G-e$ contains a good vertex.

	We next describe two configurations that cannot appear in $G$.
	\begin{itemize}
		\item[(A)] \emph{There is no vertex of degree $1$ in the outer face of $G$.}
		
		\smallskip
		Suppose for a contradiction that $u$ is a vertex of degree $1$ in the outer face
		and let $e=uv$ be the edge incident with $u$.
		Let $G'$ be obtained from $G$ by removing $u$,
		and let $\mathcal{F}'$ be obtained from $\mathcal{F}$ by including any facially adjacent pair of edges
		in $G'$ that are not facially adjacent in $G$.
		By the minimality of $G$, there exists an $\mathcal{F}'$-facial edge-coloring $c'$ of $G'$.
		Since $e$ is facially adjacent to at most two edges in $G$, there is at least one available color in $\{1,2,3\}$.
		Hence, $c'$ can be extended to an $\mathcal{F}$-facial edge-coloring of $G$, a contradiction.

		\item[(B)]  \emph{There is no edge $e$ in the outer face joining a good vertex $u$ with a vertex $v$
		such that $u$ and $v$ are in the same block, $v$ is incident with an edge $f$ that is not in the outer face,
		$f$ is facially adjacent with $e$, and $e$ is in a good pair with some edge incident to $u$ (see Figure~\ref{fig:remove}).}
		\begin{figure}[htp!]
			\begin{center}
			\includegraphics{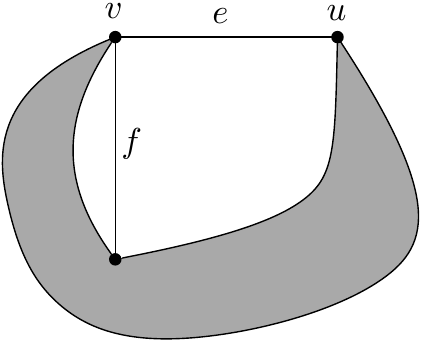}
			\end{center}
			\caption{Situation in the configuration (B) in Lemma~\ref{lem:maine}.}
			\label{fig:remove}
		\end{figure}		
		
		\smallskip
		Suppose for a contradiction that there exists such an edge $e$ in $G$.
		Let $G'$ be obtained from $G$ by removing the edges $e$ and $f$  		
		and let $\mathcal{F}'$ be obtained from $\mathcal{F}$ by including any facially adjacent pair of edges
		in $G'$ that are not facially adjacent in $G$.
		By the minimality of $G$, there exists an $\mathcal{F}'$-facial edge-coloring $c'$ of $G'$.		
		Notice that the edges of both faces with which $f$ is incident in $G$ become incident with the outer face of $G'$.
		Hence, setting $c'(f) = 4$ does not create any conflict with the other edges and it is the unique maximal color for the two faces in $G$.
		Since $e$ is in a good pair at $u$, there is at most one facially adjacent edge with $e$ at $u$ in $G$.
		There might be two facially adjacent edges with $e$ at $v$, but one of them is $f$ and as $c'(f)=4$,
		there is a color in $\{1,2,3\}$ for $e$ that is not conflicting with the edges that are facially adjacent with $e$.
		This gives a contradiction.
	\end{itemize}

	Now, let $B$ be a leaf-block in $B(G)$.
	Hence, there is at most one vertex $v \in V(B)$ with neighbors in $V(G) \setminus V(B)$,
	and it contains at least one good vertex by assumption.
	Observe that if $B$ contains an edge not incident with the outer face, then a configuration described in (B) would occur.
	Thus we may assume that every edge in $B$ is incident with the outer face. Furthermore, by (A), $B$ is a cycle.

	Let $G'$ be the graph obtained from $G$ by removing all the edges of $B$
	and let $\mathcal{F}'$ be obtained from $\mathcal{F}$ by
	including any facially adjacent pairs of edges in $G'$ that are not facially adjacent in $G$.

	By the minimality of $G$, there exists an $\mathcal{F}'$-facial edge-coloring $c'$ of $G'$
	satisfying the assumptions of the lemma.
	Now we show that $c'$ extends to $G$.
	Since $B$ is a cycle, it bounds some inner face which thus needs a unique maximal color.
	This is achieved by coloring exactly one edge of $B$ by the color $3$ and all the other edges by $1$ and $2$.

	Let $e_1$ and $e_2$ be the edges of $B$ incident with $v$. They may be facially adjacent in $G$ to edges of $G'$ that are colored by $c'$.
	Hence, each of $e_1$ and $e_2$ has two available colors and the other edges of $B$ have three available colors.
	If the color $3$ is available on $e_i$ for some $i \in \{1,2\}$, we assign $c'(e_i)=3$,
	and the remaining edges of $B$ can be colored greedily starting from $e_{3-i}$ using only the colors $1$ and $2$, a contradiction.
	Hence both, $e_1$ and $e_2$, have only the colors $1$ and $2$ available.
	Now, $B$ can be colored by coloring any edge except $e_1$ and $e_2$ by $3$ and the remaining edges of $B$, including $e_1$ and $e_2$,
	by alternating the colors $1$ and $2$. This gives a contradiction establishing Lemma~\ref{lem:maine}.
\end{proof}

We finish this section by presenting a proof of Theorem~\ref{thm:maine}.
\begin{proof}[Proof of Theorem~\ref{thm:maine}]
	Let $G$ be a $2$-(vertex-)connected plane graph.
	Let $e=uv$ be any edge in the outer face of $G$.
	Let $G'$ be the graph obtained from $G$ by removing $e$, and let $\mathcal{F}'$ be the set of facially adjacent pairs of edges in $G'$
	that are not facially adjacent in $G$.
	Notice that each of $u$ and $v$ is a good vertex in $G'$.
	Since $G$ is $2$-connected, the block graph of $G'$ is a path with $u$ and $v$ contained in the blocks (or the only block in the case when $G'$ is also $2$-connected)
	corresponding to the endvertices of the path.
	Hence, $G'$ and $\mathcal{F}'$ satisfy the assumptions of Lemma~\ref{lem:maine} and there exists
	an $\mathcal{F}'$-facial edge-coloring $c'$ of $G'$, which can be extended to a FUM-edge-coloring of $G$ by setting $c'(e) = 4$.
\end{proof}

\section{Concluding remarks}

For both variants of FUM-colorings, vertex and edge, the proposed upper bound is set at $4$ colors.
We have shown that there is no analogy with proper colorings, where some subclasses of plane graphs require at most $3$ colors.
On the other hand, we have not been able to disprove any of the two conjectures.

Although the problem of FUM-coloring is intriguing already in the class of plane graphs,
the concept can be naturally studied also for graphs embedded in higher surfaces. Youngs~\cite{You96}
proved that the chromatic number of any quadrangulation of the projective plane is either $2$ or $4$.
In Figure~\ref{fig:projective}, we present an example of projective plane graph needing $5$ colors (we leave the proof to the reader).
\begin{figure}[htp!]
	\begin{center}
		\includegraphics{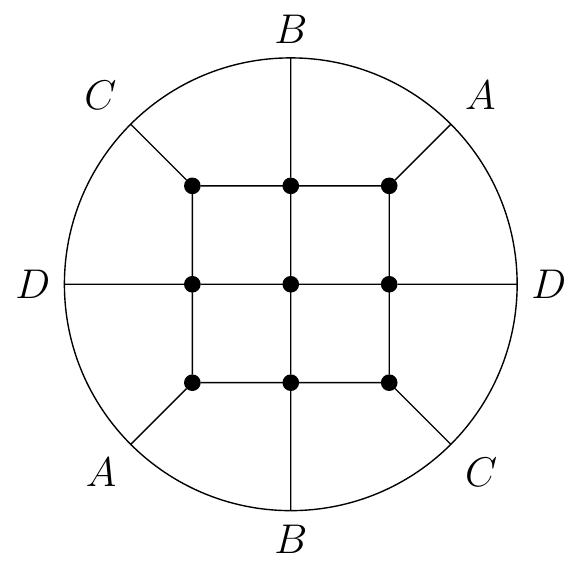}
	\end{center}
	\caption{Projective quadrangulation needing $5$ colors for a FUM-coloring.}
	\label{fig:projective}	
\end{figure}
One may therefore ask, what is the FUM chromatic number of graphs embedded in higher surfaces?
How does it behave if we add assumption on minimum face length or girth?

In~\cite{Wen16}, the author studied the list version of the problem, and he showed that having lists of size $7$
suffices for FUM-coloring of any plane graphs. He proposed the following conjecture.
\begin{conjecture}[Wendland~\cite{Wen16}]
	If each vertex of a plane graph is assigned a list of $5$ integers, then there
	exists a FUM-coloring assigning each vertex a color from its list.
\end{conjecture}
We believe that in FUM-edge-coloring, the upper bound for the list version is the same as for the ordinary.
\begin{conjecture}
	If each edge of a plane graph is assigned a list of $4$ integers, then there
	exists a FUM-edge-coloring assigning each edge a color from its list.
\end{conjecture}

\paragraph{Acknowledgment.}

The project has been supported by the bilateral cooperation between USA and Slovenia, project no. BI--US/17--18--013.
V. Andova, B. Lu\v{z}ar, and R. \v{S}krekovski were partially supported by the Slovenian Research Agency Program P1--0383.
B. Lidick\'y was partially supported by NSF grant DMS-1600390.

\end{document}